\documentclass{amsart}

\newtheorem{thm}{Theorem}
\newtheorem{lem}[thm]{Lemma}
\newtheorem{cor}[thm]{Corollary}
\newtheorem{prop}[thm]{Proposition}

\theoremstyle{definition}

\newtheorem{ex}{Example}

\providecommand{\NN}{\mathbb{N}}
\providecommand{\ZZ}{\mathbb{Z}}
\providecommand{\QQ}{\mathbb{Q}}
\providecommand{\RR}{\mathbb{R}}
\providecommand{\CC}{\mathbb{C}}

\DeclareMathOperator{\n}{n}
\DeclareMathOperator{\I}{I}
\DeclareMathOperator{\B}{B}
\DeclareMathOperator{\A}{A}
\DeclareMathOperator{\arch}{Arch}
\DeclareMathOperator{\sym}{Sym}
\DeclareMathOperator{\str}{Int}
\DeclareMathOperator{\inv}{inv}

\DeclareMathOperator{\rep}{Rep}
\DeclareMathOperator{\irr}{Irr}

\DeclareMathOperator{\mat}{Mat}

\begin{document}

\title[Archimedean quadratic modules on $\ast$-rings]
{A representation theorem for archimedean quadratic modules on $\ast$-rings}

\date{April 2005}

\author{Jakob Cimpri\v{c}}
 
\begin{abstract}
We present a new approach to noncommutative real algebraic geometry
based on the representation theory of $C^\ast$-algebras. 
An important result in commutative real algebraic geometry is 
Jacobi's representation theorem  for archimedean quadratic modules  
on commutative rings, \cite[Theorem 5]{jacobi}. 
We show that this theorem is a consequence of the 
Gelfand-Naimark representation theorem for commutative $C^\ast$-algebras.
A noncommutative version of Gelfand-Naimark theory was studied by
I. Fujimoto in \cite{fuji1,fuji2,fuji3}. We use his results to generalize
Jacobi's theorem to associative rings with involution.
\end{abstract}


\address{Cimpri\v c Jakob, University of Ljubljana, Faculty of Mathematics and Physics,
Department of Mathematics, Jadranska 19, SI-1000 Ljubljana, Slovenija}
\address{e-mail: \texttt{jaka.cimpric@fmf.uni-lj.si}}

\maketitle

\section{Introduction}

Jacobi's representation theorem \cite[Theorem 5]{jacobi} is important in the study
of positive polynomials on compact semialgebraic sets.
Its history and applications are surveyed in \cite{prestel2}. We will
\begin{itemize}
\item give a functional-analytic proof of this theorem,
\item extend it from commutative rings to noncommutative $\ast$-rings.
\end{itemize}
Our motivation comes from noncommutative real algebraic geometry; see \cite{mm2}.
We hope that this paper will convince the reader that 
irreducible $\ast$-representations should be considered 
as points of this geometry. The problem of extending 
Positivstellensatz to this context remains open.

Our work may also be of some interest to functional analysts: 
In Section \ref{env} we characterize real $C^\ast$-algebras within the class
$$\mathcal{M} = \{(A,M) \colon M \text{ is an } m\text{-admissible wedge on an involutive ring } A\}$$
and extend the notion of an enveloping $C^\ast$-algebra
from the subclass of Banach $\ast$-algebras to $\mathcal{M}$. 
In Section \ref{cp} we state and prove the real version of 
Fujimoto's CP-convexity Gelfand-Naimark theorem \cite{fuji3}.

As a motivation for later sections we present now our version of
Jacobi's representation theorem for the special case of 
\textit{commutative} $\ast$-rings.
Let $R$ be a commutative unital ring with involution $\ast$, write 
\[\begin{array}{ccc}
\sym(R) = \{a \in R \colon a=a^*\} & \text{and} &
R^+ = \{\sum_i a_i a_i^* \colon a_i \in R\}.
\end{array}\]
A subset $M$ of $\sym(R)$ is an \textit{archimedean quadratic module}
if  $-1 \not\in M$, $1 \in M$, $M+M \subseteq M$, $R^+ M \in M$ and
for every $a \in \sym(R)$ there exists $n \in \NN$ such that $n \pm a \in M$. 
Write
$$\arch(M) = \{a \in \sym(R) \colon \forall n \in \NN \ \exists k \in \NN \colon k(1+na) \in M\}.$$
The conjugation  $\phi \mapsto \overline{\phi}$,
$\overline{\phi}(a)=\overline{\phi(a)}$ is an automorphism of order 2 on
$$X_M = \{ \phi \colon R \to \CC \colon \phi \text{ a} \ast \text{-ring homomorphism such that }
\phi(M) \ge 0\}.$$
We equip $X_M$ with the topology of pointwise convergence. Finally, let
$$\mathcal{C}(X_M,-) = \{f \in \mathcal{C}(X_M,\CC) \colon f(\overline{\phi}) = \overline{f(\phi)}
\text{ for every } \phi \in X_M\}$$
with the natural involution $f \mapsto f^\ast$, $f^\ast(\phi) = f(\overline{\phi})$. 
In the original Jacobi's theorem $\ast=$ identity, and hence
$\mathcal{C}(X_M,-) = \mathcal{C}(X_M,\RR)$ and all elements of $X_M$ are real valued.

\begin{thm}
\label{rep}
Let $M$ be an archimedean quadratic module 
on a commutative unital $\ast$-ring $R$. Then the space $X_M$
is nonempty and compact. Moreover, the mapping 
$$\Phi \colon R \to \mathcal{C}(X_M,-), \quad \Phi(a)(\phi)=\phi(a),$$
is a homomorphism of unital $\ast$-rings, $\QQ \cdot \Phi(R)$ is dense in $\mathcal{C}(X_M,-)$, and
$$\Phi^{-1}(\mathcal{C}^+(X_M,-)) = \arch(M).$$
\end{thm}

\begin{proof}
Let $R$ and $M$ be as above. For every $a \in R$ write
$$\n_M(a) = \inf\{\frac{r}{s} \colon r,s \in \NN, \ r^2-s^2 aa^* \in M\}.$$
We will prove in Section \ref{env} that $\I(M) = \{a \in R \colon \n_M(a)=0\}$ 
is a $\ast$-ideal of $R$ and that $\n_M$ induces a norm on $R/\I(M)$.
Moreover, the completion $R_M$ of $R/\I(M)$ in this norm is an abelian  real 
$C^\ast$-algebra. Also, the canonical mapping $j \colon R \to R_M$ is a
homomorphism of $\ast$-rings and $j^{-1}((R_M)^+) = \arch(M)$.

Let $Y_M$ the set of all real $\ast$-algebra homomorphisms $R_M \to \CC$
with the topology of pointwise convergence.
We will see in Section \ref{up} that the mapping $Y_M \to X_M$, 
$\psi \mapsto \psi \circ j$ has an inverse $r \colon X_M \to Y_M$,
which factors an element of $X_M$ through $R/\I(M)$
and extends it by continuity to an element of $Y_M$.
The mapping $r$ is a homeomorphism with respect to the topologies of
pointwise convergence on $X_M$ and $Y_M$ and it commutes 
with the conjugations on $X_M$ and $Y_M$. It induces a mapping
$\tilde{r} \colon \mathcal{C}(Y_M,-) \to \mathcal{C}(X_M,-), f \mapsto f \circ r,$
which is one-to-one and onto, an isometry, and satisfies
$\tilde{r}^{-1}(\mathcal{C}^+(X_M,-)) = \mathcal{C}^+(Y_M,-)$.

Note that $Y_M$ coincides with the spectral space $\Omega(R_M)$;
see \cite[Definition 2.7.1, Theorem 5.2.10 and Theorem 3.2.3 $(7) \Rightarrow (4)$]{li}.
Since $\Omega(R_M)$ is  nonempty by \cite[Theorem 2.7.3]{li} and 
compact by \cite[Theorem  2.7.2 $(4)$]{li}, so also are $X_M$ and $Y_M$.
The Gelfand transform
$\Gamma \colon R_M \to \mathcal{C}(Y_M,-), 
\Gamma(a)(\psi) = \psi(a)$
is a $\ast$-isomorphism by \cite[Proposition 5.1.4]{li} and satisfies 
$\Gamma^{-1}(\mathcal{C}^+(Y_M,-))=(R_M)^+$ by \cite[Proposition 5.2.2 $(3)$
and Theorem 2.7.2 $(4)$]{li}. 

The mapping $\Phi$ can be decomposed as $\Phi = i \circ \Gamma \circ \tilde{r}$.
Since $j, \Gamma, \tilde{r}$ are homomorphisms, so is $\Phi$.
Since $\QQ \cdot j(R)$ is dense in $R_M$ and $\Gamma, \tilde{r}$ are isometries,
it follows that $\QQ \cdot \Phi(R) = \tilde{r}(\Gamma(\QQ \cdot j(R)))$ is dense in
$\mathcal{C}(X_M,-)$. 
Since $\tilde{r}^{-1}(\mathcal{C}^+(X_M,-)) = \mathcal{C}^+(Y_M,-)$,
$\Gamma^{-1}(\mathcal{C}^+(Y_M,-))=(R_M)^+$ and $j^{-1}((R_M)^+) = \arch(M)$,
it follows that $\Phi^{-1}(\mathcal{C}^+(X_M,-)) = \arch(M)$.
\end{proof}

The main difference in the noncommutative case is that we replace homomorphisms
by topologically irreducible representations on a Hilbert space of a sufficiently high dimension.
A noncommutative version of Gelfand's theory is provided by Fujimoto's CP-convexity theory.
In Section \ref{comments} we shall compare our theory with the theory
of $\ast$-orderings on $\ast$-rings. Recent generalizations of Jacobi's theorem
by M. Marshall \cite[Theorem 2.3]{mrep} and I. Klep \cite{klep} are not considered here.

\section{Quadratic modules, definition and examples}
\label{defs}

Let $A$ be a unital ring with involution and $\sym(A) = \{a \in A \vert a=a^*\}$.
A subset $M \subset \sym{A}$ is called a \textit{quadratic module} if
\begin{enumerate}
\item $-1 \not\in M$,  
\item $1 \in M$,
\item $M+M \subseteq M$,
\item $aMa^* \subseteq M$ for every $a \in A$.
\end{enumerate}
In \cite{sch}, the term \textit{$m$-admissible wedge} is used.
If $\ast=$ identity, then our definition coincides with the definition of a
quadratic module in \cite{prestel}.

Write $A^+$ for the set of all finite sums $\sum_i a_i a_i^*$.
This is consistent with the notation $\ZZ^+,\QQ^+,\RR^+,\CC^+$.
Clearly, $A^+ \subseteq M$ for every quadratic module $M$. Thus:

\begin{lem} 
The following are equivalent:
\begin{enumerate}
\item $-1 \not\in A^+$,
\item $A^+$ is a quadratic module on $A$,
\item $A$ has at least one quadratic module.
\end{enumerate}
\end{lem}

A quadratic module $M$ on $A$ is \textit{archimedean} if for every 
$a \in A$ there exist $n \in \NN$ such that $n - aa^* \in M$.

\begin{ex}
If $A = \RR[X_1,\ldots,X_n]$ with $\ast=$ identity
then $-1 \not\in A^+$.
The quadratic module $A^+$ is not archimedean. 
A quadratic module $M \subset A$ is archimedean
if and only if there exists $m \in \NN$ such that 
$m - \sum_{i=1}^n X_i^2 \in M$; see \cite[5.2.2: Putinar's criterion]{mm1}.
\end{ex}

\begin{ex}
Let $A$ be a real or complex Banach $\ast$-algebra.
Then $A^+$ is an archimedean quadratic module on $A$.
\end{ex}

\begin{ex}
\label{group}
Let $A = k[G]$ where $G$ is any group and $k$ is $\QQ$, $\RR$ or $\CC$.
For every element $a = \sum_i \alpha_i g_i \in A$ write
$$a^* = \sum_i \overline{\alpha}_i g_i^{-1}, \quad
\Vert a \Vert_1 = \sum_i \vert \alpha_i \vert.$$
Clearly, $a \mapsto a^*$ is an involution on $A$ 
and $\Vert \cdot \Vert_1$ is a norm on the $\ast$-ring $A$. Since
$$\Vert a \Vert_1^2 -aa^* = \sum_{i < j} \vert \alpha_i \alpha_j \vert 
\big(1-\frac{\alpha_i \overline{\alpha}_j}{\vert \alpha_i \alpha_j \vert}g_i g_j^{-1} \big)
\big(1-\frac{\alpha_i \overline{\alpha}_j}{\vert \alpha_i \alpha_j \vert}g_i g_j^{-1} \big)^* \in M$$
for every $a \in A$,  $A^+$ is an archimedean quadratic module on $A$. 
\end{ex}

Finally, we have several general constructions for producing
new quadratic modules from old ones.

\begin{ex}
\label{gen}
For every quadratic module $M$ on $A$ and for every subset $S \subset \sym(A)$ write 
$$M_S := \{\sum_{i,j} a_{ij} c_i a_{ij}^* \colon a_{ij} \in A, c_i \in M \cup S\}.$$
Note that $M(S)$ is a quadratic module if and only if $-1 \not\in M_S$. In this case
$M_S$ is the smallest quadratic module which contains $M$ and $S$.
\end{ex}

\begin{ex}
\label{qq}
Let $M$ be a quadratic module in $A$. Then
$$M^e := \{a \in A \colon ka \in M \mbox{  for  some } k \in \NN \}$$
is a quadratic module on $A$, 
$$M \otimes \QQ^+ := \{\sum_i m_i \otimes r_i \colon  m_i \in M, r_i \in \QQ^+\}$$
is a quadratic module on $A \otimes \QQ$, and
$$(M \otimes \QQ^+) \cap A = M^e.$$
This example shows that we may always assume without loss of generality that
$\QQ \subset A$ and $M = M^e$. (This works even if $(A,+)$ has nonzero torsion.)
\end{ex}

\begin{ex} 
\label{com}
Let $A$ be a unital $\ast$-ring. The complexification $A^\circ$ of $A$
is the set $A \times A$ with the following operations:
\begin{enumerate}
\item $(x,y)+(u,v) = (x+u,y+v)$, 
\item $-(x,y) = (-x,-y)$,
\item $(x,y)(u,v) = (xu-yv,xv+yu)$,
\item $(x,y)^* = (x^*,-y^*)$.
\end{enumerate}
Note that $A^\circ$ is also a unital $\ast$-ring with unit $(1,0)$.
The element $i = (0,1)$ behaves as imaginary unit.

Let $M$ be a quadratic module on $A$. Define 
$$M^\circ := \{\sum_i (a_i,b_i) (m_i,0) (a_i,b_i)^* \colon a_i,b_i \in A, m_i \in M\}.$$
Note that $M^\circ$ is a quadratic module on $A^\circ$.
\end{ex}

\begin{ex}
\label{mat}
Let $A$ be a unital $\ast$-ring and $n \in \NN$. The set $\mat_n(A)$
of all $n \times n$ matrices with entries in $A$ is a unital $\ast$-ring 
with involution $[a_{ij}]^* = [a_{ji}^*]$.

Let $M$ be a quadratic module on $A$. We define 
$$M_n := \{ \sum_j \begin{bmatrix} a_{1j} \\ \vdots \\ a_{nj} \end{bmatrix} m_j
\begin{bmatrix} a_{1j}^* &  \cdots & a_{nj}^* \end{bmatrix} \colon
m \in M, a_{ij} \in A\}.$$
Clearly, $M_n$ is a quadratic module on $\mat_n(A)$.
\end{ex}

\section{The $C^*$-algebra of an archimedean quadratic module}
\label{env}

From now on we assume that every $\ast$-ring is unital and contains $\QQ$.

\begin{lem}
\label{c2c}
Let $M$ be a quadratic module on a $\ast$-ring $A$.
For every $c \in \sym(A)$ and every $r \in \QQ^+$ we have
$r^2-c^2 \in M$ if and only if $r \pm c \in M$.
\end{lem}

\begin{proof}
If $r^2-c^2 \in M$, then 
$$r \pm c = \frac{1}{2r}\big((r \pm c)^2+(r^2-c^2)\big) \in M.$$
If $r \pm c \in M$, then
$$r^2-c^2 = \frac{1}{2r} \big( (r-c)(r+c)(r-c)+ (r+c)(r-c)(r+c) \big) \in M.$$
\end{proof}

For every element $a \in A$ write
$$\n_M(a) = \inf\{r \in \QQ^+ \colon r^2-aa^* \in M\}.$$
We use the convention $\inf \emptyset = \infty$.

\begin{thm}
\label{norm}
Let $M$ be a quadratic module on a $\ast$-ring $A$ and $\n = \n_M$.
For every $a, b \in A$ and every $t \in \QQ$ we have
\begin{enumerate}
\item $\n(ta) = \vert t \vert \n(a)$,
\item $\n(a) = \n(a^*)$,
\item $\n(ab) \le \n(a)\n(b)$,
\item $\n(a+b) \le \n(a)+\n(b)$,
\item $\n(aa^*) = \n(a)^2$,
\item $\n(a)^2 \le \n(aa^*+bb^*)$.
\end{enumerate}
If there exists an element $i$ in the center of $A$ such that 
$i^*=-i$ and $i^2=-1$, then the assertion (1) holds for every $t \in \QQ(i)$.
\end{thm}

\begin{proof}
The assertion (1) is trivial and assertion (5)  is a consequence of Lemma \ref{c2c}. 
To prove the assertion (2), it suffices to show that $\n(a^*) \le \n(a)$ for every 
$a \in A$. This is clear if $\n(a)=\infty$. Otherwise pick any $r \in \QQ^+$
such that $\n(a)<r$. Since
$$\big( \frac{r^2}{2} \big)^2 - \big(\frac{r^2}{2} - a^*a \big)^2
= a^*(r^2-aa^*)a \in M,$$
it follows that $\frac{r^2}{2} \pm \big(\frac{r^2}{2} - a^*a\big) \in M$ by Lemma \ref{c2c}.
Hence $\n(a^*) \le r$.

The ssertions (3) and (4) are true if either $\n(a) = \infty$ or $\n(b) = \infty$. 
Otherwise, pick any $r,s \in \QQ^+$ such that $\n(a)<r$ and $\n(b)<s$. 
Since $r^2-aa^* \in M$ and $s^2-bb^* \in M$, it follows that 
$$r^2 s^2 - (ab)(ab)^* = s(r^2-aa^*)s+a(s^2-bb^*)a^* \in M,$$
so that $\n(ab) \le rs$, proving (3). Since $\n(ab^*) < rs$ and $\n(ba^*) < rs$
by assertions (2) and (3), we have that
$$4r^2s^2 - (ab^*+ba^*)^2 = 2(r^2s^2-ab^*ba^*) + 
2(r^2 s^2 - ba^*ab^*)+(ab^*-ba^*)(ab^*-ba^*)^* \in M.$$
As  $2rs \pm (ab^*+ba^*) \in M$ by Lemma \ref{c2c}, we get
$$(r+s)^2-(a \pm b)(a \pm b)^* = r^2 - aa^*+s^2-bb^*
+ 2rs \pm (ab^*+ba^*) \in M.$$  So, $\n(a \pm b) \le r+s$, proving (4).

If $\n(aa^*+bb^*) < r$ for some $r$, then $r-aa^*-bb^* \in M$ by Lemma \ref{c2c}.
Since $bb^* \in M$, it follows that $r-aa^* \in M$. Therefore $\n(a) \le \sqrt{r}$, proving (6).
\end{proof}

Let us say that an element $a \in A$ is \textit{bounded} with respect to $M$ if $\n_M(a) < \infty$,
and \textit{infinitesimal} with respect to $M$ if $\n_M(a) = 0$.
Write $\B(M)$ for the set of all bounded elements and $\I(M)$ for the set of all infinitesimal 
elements (of $A$ with respect to $M$).  Theorem \ref{norm} implies the following result:

\begin{cor} 
\label{cast}
Take $A$ and $M$ as above. 

$\B(M)$ is a $\ast$-subring of $A$ and  
$\I(M)$ is a two-sided $\ast$-ideal in $\B(M)$.

The mapping $\n_M$ induces a norm $\Vert \cdot \Vert$ on $\B(M)/\I(M)$. 
Denote by $A_M$ the completion of $\B(M)/\I(M)$ with respect to this norm. 
Then $A_M$ is a real $C^*$-algebra.

If there exists an element $i$ in the center of $A$ such that 
$i^*=-i$ and $i^2=-1$, then $A_M$ is a complex $C^*$-algebra.
\end{cor}

Property (6) from Theorem \ref{norm} is very important 
in the theory of real $C^*$-algebras, 
because a $C^*$-norm with this property
extends to a $C^*$-norm on the complexification of the algebra;
see \cite{palmer}. The spectral and representation theory of 
such real $C^*$-algebras work as in the complex case;
we refer to \cite{const2,const3} or \cite{li}.

\begin{ex}
\label{pos1}
Let $A$ be either a real $C^*$-algebra with the property 
$\Vert a \Vert^2 \le \Vert aa^*+bb^* \Vert$ for all $a,b \in A$
or a complex $C^*$-algebra. If $M = A^+$ then
$$\Vert a \Vert = \n_M(a) \mbox{ for every } a \in A,$$
so that $A = A_M$. Namely, \cite[Corollary 4.2.1.16]{const3}
says that for every $x \in \sym(A)$,
$\Vert x \Vert \le r$ if and only if $r1 \pm x \in A^+$
if and only if $\sigma(x) \in [-r,r]$. 
\end{ex}

\begin{ex}
If $M^\circ$ is as in Example \ref{com}, then $\big( A^\circ \big)_{M^\circ} \cong (A_M)^\circ$.
If $M_n$ is as in Example \ref{mat}, then $\mat_n(A)_{M_n} \cong \mat_n(A_M)$. 
We omit the proofs because they are straightforward 
and because we will not use these results in the sequel.
\end{ex}

For every archimedean quadratic module $M$ on a $\ast$-ring $A$ the seminorm $\n_M$ 
defines a topology on $A$ with basis $B(a,\epsilon) = \{b \in A \colon \n_M(b-a)<\epsilon\}$.
Write $\arch(M)$ for the $\n_M$-closure of $M$ and $\str(M)$ for the $\n_M$-interior
of $M$ in $A$. Note:

\begin{lem} 
\label{pos2} 
Let $A$, $M$, $\n_M$ be as above and $x \in \sym(A)$. 
The following properties of $x$ are equivalent:
\begin{enumerate}
\item $x \in \arch(M)$,
\item $\n_M(r-x) \le r$ for some $r \in \QQ^{>0}$ such that $r \ge \n_M(x)$,
\item $r+x \in M$ for every $r \in \QQ^{>0}$.
\end{enumerate}
Similarly, the following properties of $x$ are also equivalent:
\begin{enumerate}
\item $x \in \str(M)$,
\item $\n_M(r-x) < r$ for some $r \in \QQ^{>0}$ such that $r \ge \n_M(x)$,
\item $x \in r+M$ for some $r \in \QQ^{>0}$.
\end{enumerate}
\end{lem}

The following result is useful:

\begin{thm}
\label{stone}
Let $A$, $M$ be as above and denote by $j \colon A \to A_M$ the canonical mapping.
For every $x \in \sym(A)$,
\begin{enumerate}
\item $x \in \arch(M)$ if and only if $j(x) \in (A_M)^+$,
\item $x \in \str(M)$ if and only if $j(x) \in (A_M)^+ \cap \inv(A_M)$.
\end{enumerate}
\end{thm}

\begin{proof}
To prove the assertion (1) pick any $x \in \sym(A)$. By Lemma \ref{pos2}, $x \in \arch(M)$ 
if and only if $\n_M(r-x) \le r$ for some rational $r \ge \n_M(x)$.
Since $\n_M(a) = \Vert j(a) \Vert$ for every $a \in A$,
$\n_M(r-x) \le r$ is equivalent to $\Vert r -j(x) \Vert \le r$.
By Example \ref{pos1} and Lemma \ref{pos2}, this is
equivalent to $j(x) \in (A_M)^+$. The assertion (2) is similar.
\end{proof}

Theorem \ref{stone} implies the following generalization to $C^\ast$-algebras
of the famous Stone's characterization of rings of continuous functions \cite{st}.

\begin{cor}
Let $M$ be a quadratic module on a $\ast$-ring $A$. 
Then $A$ is a $C^\ast$-algebra with positive cone $M$ if and only if
\begin{enumerate}
\item $M \cap -M = \{0\}$,
\item $M$ is archimedean, i.e. $B(M) = A$,
\item $M = \arch(M)$,
\item $A$ is complete in the norm $\n_M$.
\end{enumerate}
\end{cor}

The following two examples will follow from Theorem \ref{int} in Section \ref{up}.

\begin{ex}
If $A$ is a real or complex unital Banach $\ast$-algebra and $M = A^+$ then
$A_M$ is exactly the $C^\ast$-enveloping of $A$; see \cite[2.7.2]{dix}.
Namely, by the assertion (1) of Theorem \ref{int}, $\n_M$ coincides with the $C^\ast$-seminorm
$\Vert \cdot \Vert'$ in the sense of \cite[Proposition 2.7.1]{dix}.
\end{ex}

\begin{ex}
Let $G$ be any group, and denote by $\CC[G]$ its group ring and by $L^1(G)$
the completion of $\CC[G]$ in the norm $\Vert \cdot \Vert_1$
of Example \ref{group}. Note that $L^1(G)$ is an involutive
complex Banach algebra. Its enveloping $C^*$-algebra is denoted by $C^*(G)$
and called \textit{the $C^*$-algebra of $G$}; see \cite[Section 13.9]{dix}.
If $A = \CC[G]$ and $M = A^+$, then $A_M = C^*(G)$.
\end{ex}

\section{$M$-positive mappings}
\label{up}

A \textit{positive form} on a $\ast$-ring $A$
is a mapping $f \colon A \to \CC$ such that
$f(a+b) = f(a)+f(b)$, $f(a^*) = \overline{f(a)}$
and $f(aa^*) \ge 0$ for every $a,b \in A$.

\begin{prop}
\label{formeq}
Let $M$ be an archimedean quadratic module on a $\ast$-ring $A$.
For every positive form $f$ on $A$, the following properties are equivalent:
\begin{enumerate}
\item $f(M) \ge 0$.
\item $\vert f(s) \vert \le \n_M(s) f(1)$ for every $s \in \sym(A)$,
\item $\vert f(a) \vert \le \n_M(a) f(1)$ for every $a \in A$.
\end{enumerate}
\end{prop}

\begin{proof}
Assume that (1) is true and pick $s \in \sym(A)$. 
For every $r \in \QQ^+$ such that $\n_M(s) < r$, we have that $r^2 - s^2 \in M$, 
hence $r \pm s \in M$ by Lemma \ref{c2c}. Since $f(M) \ge 0$, it follows that 
$r f(1) \pm f(s) = f(r \pm s) \ge 0$, hence $\vert f(s) \vert \le r f(1)$.
Therefore (2) is true.
Conversely, if (2) is true,  pick $m \in M$ and $r \in \QQ^+$
such that $\n_M(m) < r$. By Lemma \ref{pos2}, $\n_M(r-m) \le r$. 
By (2), $\vert r f(1) - f(m) \vert \le \n_M(r -m) f(1)$. 
It follows that $f(m) \ge 0$. Hence (1) is true.

Assume now that (2) is true and pick $a \in A$. By the Cauchy-Schwartz inequality,
we have $\vert f(a) \vert^2 \le f(aa^*)f(1)$. Applying (2) with $s = aa^*$,
we get $f(aa^*) \le \n_M(aa^*) f(1)$. Finally $\n_M(aa^*) = \n_M(a)^2$
by Theorem \ref{norm}. It follows that (2) is true. Clearly, (3) implies (2).
\end{proof}

Let $A$ be a $\ast$-ring and $H$ a complex Hilbert space.
A \text{representation} of $A$ on $H$ is a (non-unital) 
homomorphism of $\ast$-rings. Let us say that a representation $\psi$ of $A$ on $H$
is $M$-\textit{positive} if $\pi(m)$ is positive semidefinite for every $m \in M$.

\begin{prop}
\label{cpeq}
Let $M$ be an archimedean quadratic module on $\ast$-ring $A$
and $H$ a complex Hilbert space. Then every $M$-positive representation
$\psi$ of $A$ on $H$ satisfies $\Vert \psi(a) \Vert \le \n_M(a) \Vert \psi(1) \Vert$.
\end{prop}

\begin{proof}
Pick $\psi \in  \rep_\ZZ^M(A,H)$. For every $\xi \in H$ and 
$a \in A$ write $f_\xi(a) = \langle \psi(a) \xi,\xi \rangle$.
Clearly, each $f_\xi$ is a positive form and $f_\xi(M) \ge 0$.  
By Proposition \ref{formeq}, $|f_\xi(s)| \le n_M(s) f_\xi(1)$
for every $s \in A$. It follows that for every $a \in A$,
$\Vert \psi(a) \Vert^2 = 
\sup_\xi \frac{\langle \psi(a) \xi, \psi(a) \xi \rangle}{\langle \xi, \xi \rangle} =
\sup_\xi \frac{\langle \psi(a^*a) \xi, \xi \rangle}{\langle \xi, \xi \rangle} \le
\n_M(a^*a)\sup_\xi \frac{\langle \psi(1) \xi, \xi \rangle}{\langle \xi, \xi \rangle} =
\n_M(a)^2 \Vert \psi(1) \Vert$.
\end{proof}

A representation $\psi$ of a $\ast$-ring $A$ on a complex Hilbert space $H$
is \textit{irreducible} (resp. \text{cyclic}) if $\psi(A) \xi$ is dense in 
$H_\psi := \overline{\psi(A)H}$ for every (resp. for some) $\xi \in H$.

\begin{lem}
\label{formup}
Let $M$ be a quadratic module on a $\ast$-ring $A$.
For a complex Hilbert space $H$ there are natural one-to-one  correspondences between
\begin{enumerate}
\item the set $\rep_\ZZ^M(A,H)$ of all $M$-positive representations of $A$ on $H$,
\item the set $\rep_\RR(A_M,H)$ of all $\RR$-linear representations of $A_M$ on $H$,
\item the set $\rep((A_M)^\circ,H)$ of all $\CC$-linear representations of $(A_M)^\circ$ on $H$.
\end{enumerate}
The correspondences preserve the property of being irreducible or cyclic.
\end{lem}

\begin{proof}
Every $M$-positive representation of $A$ on $H$ is continuous 
with respect to $\n_M$ by Proposition \ref{cpeq}.
Hence, it can be factored through $A/\I(M)$ and then extended by continuity to $A_M$.
The continuity implies that the extension to $A_M$ is $\RR$-linear. 
The converse mapping is given by $\psi \mapsto \psi \circ j$; see Theorem \ref{stone}.

Every $\RR$-linear representation $\psi$ of $B = A_M$ on $H$ extends to a $\CC$-linear 
representation $\psi^\circ$ of $B^\circ$ on $H$ by $\psi^\circ(b',b'') = \psi(b')+i \psi(b'')$ 
for every $b',b'' \in B$. The converse mapping is the restriction mapping $\pi \mapsto \pi|_B$. 
\end{proof}

Write $\irr_\ZZ^M(A,H)$, $\irr_\RR(A_M,H)$ and $\irr((A_M)^\circ,H)$ 
for the corresponding sets of irreducible representations. 
Write $\irr_\ZZ^M(A) = \bigcup_H \irr_\ZZ^M(A,H)$ 
where $H$ runs through all complex Hilbert spaces.

\begin{thm}
\label{int}
Let $M$ be an archimedean quadratic module on $A$ and $a \in A$. Then:
\begin{enumerate}
\item $\n_M(a) = \sup_{\psi \in \irr_\ZZ^M(A)} \Vert \psi(a) \Vert$,
\item $a \in \arch(M)$ if and only if $\psi(a)$ is positive semidefinite for every $\psi \in \irr_\ZZ^M(A)$,
\item $a \in \str(M)$ if and only if $\psi(a)$ is positive definite  for every $\psi \in \irr_\ZZ^M(A)$.
\end{enumerate}
\end{thm}

\begin{proof}
By Lemma \ref{formup} and Theorem \ref{stone}
we may assume that $A$ is a complex $C^\ast$-algebra and $M=A^+$.
In this case, the results are known from \cite[Sections 2.6 and 2.7]{dix}. 
Namely, assertion (1) follows from \cite[2.7.1 and 2.7.3]{dix},
assertion (2) is a variant of \cite[2.6.2]{dix} which follows from 
\cite[2.5.4]{dix} and Krein-Milman Theorem and assertion (3) 
is another variant of \cite[2.6.2]{dix} which follows from
\cite[remarks after Definition 2.14.6 and Proposition 2.3.13]{li1}.
\end{proof}

Let $M$ be an archimedean quadratic module on a $\ast$-ring $A$.
Write $\alpha_i(A,M) = \sup_{\pi \in \irr_\ZZ^M(A)} \dim H_\pi$. 
Define $\alpha_c(A,M)$ similarly, just replacing irreducible by cyclic representations.
If $H$ is a complex Hilbert space with $\dim H \ge \alpha_i(A,M)$, then every 
ireducible $M$-positive representation of $A$ can be realized on $H$. For such $H$
$\irr_\ZZ^M(A)$ in Theorem \ref{int} can be replaced by $\irr_\ZZ^M(A,H)$. 
Let $\A_u^E(\irr_\ZZ^M(A,H),L(H))$ denote the set of all mappings
$\gamma \colon \irr_\ZZ^M(A,H) \to L(H)$ such that
\begin{enumerate}
\item $\gamma$ is bounded (i.e. $\Vert \gamma \Vert := 
\sup_{\pi \in \irr_\ZZ^M(A,H)} \Vert \gamma(\pi) \Vert < \infty$),
\item $\gamma$ is equivariant (i.e. $\gamma(u^* \pi u) = u^* \gamma(\pi) u$
for every $\pi \in \irr_\ZZ^M(A,H)$ and every partial isometry $u \in L(H)$ 
such that $uu^* \ge$ the projection on $H_\pi$),
\item $\gamma$ is uniformly continuous (with respect to the weak operator topology on $L(H)$
and the topology of pointwise convergence on $\irr_\ZZ^M(A,H)$).
\end{enumerate}
The set $\A^E_u(\irr_\ZZ^M(A,H),L(H))$ is a complex $C^*$-algebra for pointwise 
algebraic operations and the norm $\gamma \mapsto \Vert \gamma \Vert$.

For every $\pi \in \irr_\ZZ^M(A,H)$ define $\bar{\pi} \in \irr_\ZZ^M(A,H)$
by $\bar{\pi}(a) = \pi(a)^*$ for every $a \in A$. 
Write $\A^E_u(\irr_\ZZ^M(A,H),-)$ for the set of all $\gamma \in \A^E_u(\irr_\ZZ^M(A,H),L(H))$ 
such that $\gamma(\bar{\pi}) = \gamma(\pi)^\ast$ for every $\pi \in \irr_\ZZ^M(A,H)$. Note that 
$\A^E_u(\irr_\ZZ^M(A,H),-)$ is a real $C^\ast$-subalgebra of $\A^E_u(\irr_\ZZ^M(A,H),L(H))$. 
Its positive cone  $\A_u^E(\irr_\ZZ^M(A,H),-)^+$ is equal to the set of all
$\gamma \in A^E_u(\irr_\ZZ^M(A,H),-)$ such that $\gamma(\pi)$ is 
positive semidefinite for every $\pi \in \irr_\ZZ^M(A,H)$. 

We can rephrase Theorem  \ref{int} as a generalization of Jacobi's theorem:

\begin{thm} 
\label{rep1}
Let $M$ be an archimedean quadratic module on a $\ast$-ring $A$
and $H$ a complex Hilbert space such that $\dim H \ge \alpha_i(A,M)$.
The evaluation mapping 
\[
\Phi \colon A \to \A_u^E(\irr_\ZZ^M(A,H),L(H)), \quad \Phi(a)(\pi) = \pi(a),
\]
is a $\ast$-homomorphism and an isometry. Moreover,
\[
\arch(M) = \Phi^{-1}(\A_u^E(\irr_\ZZ^M(A,H),-)^+).
\]
\end{thm}

When we compare Theorem \ref{rep1} with Theorem \ref{rep}
the following questions arise:
\begin{enumerate}
\item Is $\QQ \cdot \Phi(A)$ dense in $\A_u^E(\irr_\ZZ^M(A,H),-)$?
\item Is $\irr_\ZZ^M(A,H)$ compact in the topology of pointwise convergence?
\end{enumerate}

The answer to question (1) is yes if $\dim H \ge  \alpha_c(A,M)$;
see Theorem \ref{mgelfand}. 
Note that the properties (2) and (3) from the definition of 
$\A_u^E(\irr_\ZZ^M(A,H),L(H))$  are not required in the proof of Theorem \ref{rep1}.
However, they will be required in the proof of Theorem \ref{mgelfand}. 
We don't know the answer to question (1) if 
$\alpha_i(A,M) \le \dim H < \alpha_c(A,M)$.

We believe that the answer to question (2) is no (cf. \cite{bicht})
but we don't have an explicit counterexample. There exists
a natural compactification of $\irr_\ZZ^M(A,H)$, namely its closure
in the set of all additive mappings $\psi \colon A \to L(H)$
of norm $\le 1$. This follows from the fact that the unit ball of 
$L(H)$ is compact in the weak operator topology.

\section{Real CP-convexity Gelfand-Naimark Theorem}
\label{cp}

The aim of this section is to prove a real version of
the CP-convexity Gelfand-Naimark theorem from \cite{fuji3}
 similar to the real Gelfand-Naimark theorem from \cite{li}.

Let $A$ be a complex $C^\ast$-algebra and $H$ a complex Hilbert space.
Let us denote by $\A^E_u(\irr(A,H),L(H))$ the set of all mappings 
$\kappa \colon \irr(A,H) \to L(H)$ which are equivariant,
bounded and uniformly continuous as above.
Let $\alpha_c(A)$ denote the supremum of $\dim H_\pi$ where 
$\pi$ runs through all cyclic representations of $A$ on all complex Hilbert spaces. 
The CP-conveity Gelfand-Naimark theorem from \cite{fuji3} says:

\begin{thm} 
\label{cgelfand}
Let $A$ be a complex $C^\ast$-algebra and $H$ a complex Hilbert space 
such that $\dim H \ge \alpha_c(A)$. The Gelfand transform
$$g \colon A \to \A^E_u(\irr(A,H),L(H)), \quad g(c)(\psi) = \psi(c),$$
is a $\ast$-isomorphism and an isometry. 
\end{thm}

Now let us turn our attention to the real case. 
Let $B$ be a real $\ast$-algebra with complexification $B^\circ$
and $H$ a complex Hilbert space. Write $A^E_u(\irr_\RR(B,H),L(H))$
for the set of all mappings $\eta \colon \irr_\RR(B,H) \to L(H)$
which are equivariant, bounded and uniformly continuous.
Let $s \colon \irr_\RR(B,H) \to \irr(B^\circ,H)$
denote the natural correspondence of Lemma \ref{formup}.
The correspondence $s$ is a homeomorphism with respect 
to the topologies of pointwise convergence. The mapping 
$$A^E_u(\irr_\RR(B,H),L(H)) \to A^E_u(\irr(B^\circ,H),L(H)), \quad \eta \mapsto s \eta s^{-1},$$
is an $\ast$-isomorphism and isometry.

For every $\rho \in \irr_\RR(B,H)$ we define $\bar{\rho} \in \irr_\RR(B,H)$
by $\bar{\rho}(b)=\rho(b)^\ast$ for every $b \in B$. Write $\A^E_u(\irr_\RR(B,H),-)$
for the set of all $\eta \in \A^E_u(\irr_\RR(B,H),L(H))$ such that 
$\eta(\bar{\rho}) = \eta(\rho)^\ast$ for every $\rho \in \irr_\RR(B,H)$.
This is a real $C^\ast$-subalgebra of $\A^E_u(\irr_\RR(B,H),L(H))$.
The mapping
$$g_\RR \colon B \to \A^E_u(\irr_\RR(B,H),-), \quad g_\RR(b)(\eta) = \eta(b),$$
will be called \textit{the real Gelfand transform}. Since 
$g_\RR(b)(\bar{\rho}) = \bar{\rho}(b) = \rho(b)^\ast = g_\RR(b)(\rho)^\ast$
for every $b \in B$, it follows that $g_\RR(b) \in \A^E_u(\irr_\RR(B,H),-)$ for every $b \in B$.
Hence $g_\RR$ is well defined. Clearly, $g_\RR$ is a homomorphism of real $\ast$-algebras.

Theorem \ref{rgelfand} is a real version of Theorem \ref{cgelfand}.

\begin{thm}
\label{rgelfand}
Let $B$ be a real $\ast$-algebra and $H$ a complex 
Hilbert space such that $\dim H \ge \alpha_c(B^\circ)$.
The real Gefand transform $g_\RR \colon B \to \A^E_u(\irr_\RR(B,H),-)$
is a $\ast$-isomorphism and an isometry.
\end{thm}

\begin{proof}
We have a commutative diagram
\[\begin{array}{ccc}
B^\circ & \stackrel{g}{\longrightarrow} & \A^E_u(\irr(B^\circ,H),L(H)) \\
\uparrow & & \uparrow \eta \mapsto s \eta s^{-1} \\
B & \stackrel{g_\RR}{\longrightarrow} & \A^E_u(\irr_\RR(B,H),-) 
\end{array}\]
where the vertical arrows are one-to-one. By Theorem \ref{cgelfand},
$g$ is one-to-one and onto. It follows that $g_\RR$ is one-to-one.
It remains to show that $g_\RR$ is onto.

For every $\pi \in \irr(B^\circ,H)$ write $\bar{\pi}$ for the mapping defined by
$\bar{\pi}(c) = \pi(\bar{c})^\ast$ for $c \in B^\circ$. Note that
$\overline{\rho^\circ} = (\bar{\rho})^\circ$ for every $\rho \in \irr_\RR(B,H)$.
It follows that
\[\begin{array}{l}
\A^E_u(\irr(B^\circ,H),-) := 
\{s \eta s^{-1} \colon \eta \in \A^E_u(\irr_\RR(B,H),-)\} 
= \\ = 
\{\gamma \in \A^E_u(\irr(B^\circ,H),L(H)) \colon \gamma(\bar{\pi}) = \gamma(\pi)^\ast 
 \text{ for all } \pi \in \irr(B^\circ,H)\}.
\end{array}\]
Pick any $c \in B^\circ$. Note that $g(\bar{c})(\pi)^\ast = g(c)(\bar{\pi})$ 
for every $\pi \in \irr(A^\circ,H)$. It follows that $g(c) \in \A^E_u(\irr(B^\circ,H),-)$
if and only if $g(c)(\pi)^\ast = g(\bar{c})(\pi)^\ast$ for every $\pi \in \irr(A^\circ,H)$.
Since $g$ is one-to-one, this is equivalent to $c = \overline{c}$. Since $g$ is onto,
it follows $g(B) = \A^E_u(\irr(B^\circ,H),-)$. Hence, $g_\RR$ is onto.
\end{proof}

The following corollary of Theorem \ref{rgelfand} complements Theorem \ref{rep1}.

\begin{thm}
\label{mgelfand}
Let $A$, $M$, $H$ and $\Phi$ be as in Theorem \ref{rep1}.
If  $\dim H \ge \alpha_c(A,M)$ then $\QQ \cdot \Phi(A)$ 
is dense in $A^E_u(\irr_\ZZ^M(A,H),-)$.
\end{thm}

\begin{proof}
Let $r \colon \irr_\ZZ^M(A,H) \to \irr_\RR(A_M,H)$ be the natural correspondence
from Lemma \ref{formup}. Clearly, $r$ is a homeomorphism with respect to the
topologies of pointwise convergence and it induces a mapping 
\[
\tilde{r} \colon  A^E_u(\irr_\ZZ^M(A,H),-) \to A^E_u(\irr_\RR(A_M,H),-), \quad
\kappa \mapsto r \kappa r^{-1},
\]
which is a $\ast$-isomorphism and an isometry. The diagram
\[\begin{array}{ccc}
A_M & \stackrel{g_\RR}{\rightarrow} & A^E_u(\irr_\RR(A_M,H),-) \\
j \uparrow &    & \uparrow \tilde{r}  \\
A & \stackrel{\Phi}{\rightarrow} & A^E_u(\irr_\ZZ^M(A,H),-) \\
\end{array}\]
is commutative. Since $\dim H \ge \alpha_c(A,M) = \alpha((A_M)^\circ)$,
Theorem \ref{rgelfand} implies that $g_\RR$ is onto. We know from
Theorem \ref{rep1} that $g_\RR$ is an isometry. It is clear from
the construction of $A_M$ in Section \ref{env} that $\QQ \cdot j(A)$
is dense in $A_M$. Since $g_\RR$ and $\tilde{r}$ are isometries and onto,
it follows that $\QQ \cdot \Phi(A) = \tilde{r}^{-1}(g_\RR(\QQ \cdot j(A)))$ 
is dense in $A^E_u(\irr_\ZZ^M(A,H),L(H)) = \tilde{r}^{-1}(g_\RR(A_M))$.
\end{proof}

\section{Comments on $\ast$-orderings}
\label{comments}

When functional analysts and real algebraic geometers talk about
ordered complex $\ast$-algebras, they don't mean the same thing.
For a functional analyst, an ordering on $A$ is a cone on $A$,
i.e. a subset $C \subset \sym(A)$ such that $C+C \subseteq C$ and $\RR^+ C \subseteq C$.
For a real algebraic geometer, an ordering on $A$ is usually a $\ast$-ordering,
i.e. a subset $P \subseteq \sym(A)$ such that $P+P \subset P$,
$aPa^* \subseteq P$ for every $a \in A$, $st+ts \in P$ for every $s,t \in P$,
$P \cap -P$ is a Jordan prime ideal and $P \cup -P = \sym(A)$; see \cite{mm2}.
Note that every $\ast$-ordering is a cone. The full matrix ring $\mat_n(\CC)$ ($n \ge 2$)
is a typical example of a complex $\ast$-algebra that is ordered for a functional analyst
and not orderable for a real algebraic geometer. Another example is group rings
$\CC[G]$ which are always orderable for a functional analyst and only in
special cases (for certain orderable groups) for a real algebraic geometer.

Let us recall the motivation for the definition of a $\ast$-ordering.
The most trivial example is $(\CC,\RR^+)$. If $A$ is a commutative 
complex $\ast$-algebra and $\phi \colon A \to \CC$ is a hermitian
homomorphism, then $P := \phi^{-1}(\RR^+) \cap \sym(A)$ is
a natural candidate for a $\ast$-ordering. We list its algebraic properties
($P+P \subseteq P$, $PP \subseteq P$, $aa^* \in P$ for every $a \in P$, 
$P \cap -P$ is a prime ideal and $P \cup -P = \sym(A)$) and take them
as axioms of a $\ast$-ordering. The noncommutative definition is 
a modification that makes most of the commutative theory work.

A definition of an ordering that is not too restrictive for functional analysts
and not too general for real algebraic geometers should follow the same
steps as in the commutative case. Let us consider the set $\Pi_n$ of all 
positive semidefinite hermitian matrices on $\mat_n(\CC)$ as the simplest ordering.
Let $A$ be a complex $\ast$-algebra, $\pi \colon A \to \mat_n(\CC)$
an irreducible $\ast$-representation and set $P = \pi^{-1}(\Pi_n) \cap \sym(A)$.
The algebraic properties of $P$ include the following:
\begin{enumerate}
\item $P+P \subseteq P$,
\item if $a,b \in P$ commute, then $ab \in P$,
\item $aPa^* \subseteq P$ for every $a \in A$,
\item $P \cap -P$ is the symmetric part of a prime ideal,
\item for every primitive hermitian idempotent $e \in A$,
$eAe$ is linearly ordered by $P \cap eAe$.
\end{enumerate}
Similar orderings have been considered in \cite{upor}.
It would be interesting to know whether an Artin-Schreier theory 
of such orderings can be developed.

\end{document}